\documentclass[11pt]{amsart}

\usepackage{amsfonts,amsmath,amssymb}

\usepackage{enumerate}

\usepackage{tikz}

\usepackage{tikz-cd}

\numberwithin{equation}{section}

\newtheorem{theorem}{Theorem}[section]

\newtheorem{lemma}[theorem]{Lemma}

\newtheorem{definition}{Definition}[section]

\newtheorem{corollary}[theorem]{Corollary}

\newtheorem{remark}[theorem]{Remark}

\newcommand{\cl}[1]{\mathcal{#1}} 

\newcommand{\bb}[1]{\mathbb{#1}}

\newcommand{\sca}[1]{\left\langle#1\right\rangle} 

\newcommand{\nor}[1]{\left\Vert #1\right\Vert}

\begin{document}

\title{A Morita characterisation for algebras and spaces of operators on Hilbert spaces}

\author[G.K. Eleftherakis and E.Papapetros ]{G.K. Eleftherakis and E. Papapetros }

\address{G. K. Eleftherakis\\ University of Patras\\Faculty of Sciences\\ Department of Mathematics\\265 00 Patras Greece }

\email{gelefth@math.upatras.gr}

\address{E. Papapetros\\ University of Patras\\Faculty of Sciences\\ Department of Mathematics\\265 00 Patras Greece }

\email{e.papapetros@upatras.gr} 




\date{}

\maketitle

\begin{abstract}We introduce the notion of $\Delta$ and $\sigma\,\Delta-$ pairs for operator algebras and characterise $\Delta-$ pairs through their categories of left operator modules over these algebras. Furthermore, we introduce the notion of $\Delta$-Morita equivalent operator spaces and prove a similar theorem about their algebraic extensions. We prove that $\sigma\Delta$-Morita equivalent operator spaces are stably isomorphic and vice versa. Finally, we study unital operator spaces, emphasising their left (resp. right) multiplier algebras, and prove theorems that refer to $\Delta$-Morita equivalence of their algebraic extensions.

\end{abstract}

\section{Introduction}

In what follows, if 
$\cl X$ is a subset of $\mathbb{B}(H_1,H_2)$ and $\cl Y$ is a subset of $\mathbb{B}(H_2,H_3)$, then we denote by $\overline{[\cl Y\cl X]}$ the norm-closure of the linear span of the set $$\left\{y\,x\in\mathbb{B}(H_1,H_3)\,,y\in \cl Y\,,x\in \cl X\right\}.$$ Similarly, if $\cl Z$ is a subset of $\mathbb{B}(H_3,H_4)$, we define the space $\overline{[\cl Z\cl Y\cl X]}.$

If $H\,,K$ are Hilbert spaces, then a linear subspace $M\subseteq \mathbb{B}(H,K)$ is called a ternary ring of operators (TRO) if $M\,M^{\star}\,M\subseteq M.$ It then follows that $M$ is an $\mathcal{A}-\mathcal{B}$ equivalence bimodule in the sense of Rieffel for the $C^{\star}$-algebras $\mathcal{A}=\overline{[M\,M^{\star}]}$ and $\mathcal{B}=\overline{[M^{\star}\,M]}$. 

We call a norm closed ternary ring of operators $M,$ $\sigma$-TRO if there exist sequences $\left\{m_i\in M\,,i\in\mathbb{N}\right\}$ and $\left\{n_j\in M\,,j\in\mathbb{N}\right\}$ such that $$\lim_{t}\sum_{i=1}^t m_i\,m_i^{\star}\,m=m\,\,,\lim_{t}\sum_{j=1}^t m\,n_j^{\star}\,n_j=m\,,\forall\,m\in M$$ and $$\nor{\sum_{i=1}^t m_i\,m_i^{\star}}\leq 1\,,\nor{\sum_{j=1}^t n_j^{\star}\,n_j}\leq 1\,,\forall\,t\in\mathbb{N}.$$ Equivalently, a TRO $M$ is a $\sigma$-TRO if and only if the $C^\star$-algebras $\overline{[M^{\star}\,M]},$ $\overline{[M\,M^{\star}]}$ have a $\sigma$-unit.

At the beginning of the 1970s, M. A. Rieffel introduced the idea of Morita equivalence of $C^{\star}$-algebras. In particular, he gave the following definitions: \\ $i)$ Two $C^{\star}$-algebras, $\mathcal{A}$ and $\mathcal{B}$, are said to be Morita equivalent if they have equivalent categories of $\star$-representations via $\star$-functors.\\ $ii)$ The same algebras are said to be strongly Morita equivalent if there exists an $\mathcal{A}-\mathcal{B}$ module of equivalence or if there exists a TRO $M$ such that the $C^{\star}$ algebras $\overline{[M^{\star}\,M]}$ and $\mathcal{A}$ (resp. $\overline{[M\,M^{\star}]}$ and $\mathcal{B}$) are $\star$-isomorphic. We write $\mathcal{A}\sim_{R}\mathcal{B}$. If $\mathcal{A}\sim_{R}\mathcal{B}$, then $\cl A$ and $\cl B$ have equivalent categories of representations. The converse does not hold. For further details, see \cite{Rie74, Rie74-2,Rie82}.

Brown, Green and Rieffel proved the following fundamental theorem for $C^{\star}$-algebras (\cite{Bro77,BGR77}).

\begin{theorem}

If $\mathcal{A}\,,\mathcal{B}$ are $C^{\star}$-algebras with $\sigma$-units, then $\mathcal{A}\sim_{R}\mathcal{B}$ if and only if they are stably isomorphic, which means that the algebras $\mathcal{A}\otimes \cl K\,,\mathcal{B}\otimes \cl K$ are $\star$-isomorphic. Here, $\cl K$ is the algebra of compact operators acting on $\ell^2(\mathbb{N})$, and $\otimes$ is the minimal tensor product.

\end{theorem}

The next step in this theory came from Blecher, Muhly and Paulsen. They defined the notion of strong Morita equivalence $\sim_{BMP}$ for operator algebras, self-adjoint or not, and they proved that if $\mathcal{A}\sim_{BMP}\mathcal{B}$, their categories of left operator modules are equivalent (\cite{BMP00}). Later, Blecher proved that the converse is also true (\cite{Ble01-2}). Therefore, he proved that two $C^\star$-algebras $\cl A, \cl B$ have equivalent categories of left operator modules if and only if $\cl A \sim_R\cl B$.

A third notion of Morita equivalence was introduced by the first author of this article. According to this theory, two operator algebras, $\mathcal{A}\,,\mathcal{B}$, are said to be $\Delta$-equivalent and we write $\mathcal{A}\sim_{\Delta}\mathcal{B}$ if they have completely isometric representations $\alpha: \cl A\rightarrow \alpha(\cl A)\subseteq \bb B(H), \;\;\beta: \cl B\rightarrow \beta(\cl B)\subseteq \bb B(K)$ and there exists a TRO $M\subseteq \bb B(H,K)$ such that $$\alpha(\cl A)=\overline{[M^{\star}\,\beta(\mathcal{B})\,M]},\;\;\beta(\cl B)=\overline{[M\,\alpha(\mathcal{A})\,M^{\star}]}$$ (\cite{Ele14}). If $M$ is a $\sigma$-TRO, we write $\mathcal{A}\sim_{\sigma\,\Delta}\mathcal{B}.$ 

G .K. Eleftherakis proved that $\mathcal{A}\sim_{\sigma\,\Delta}\mathcal{B}$ if and only if $\mathcal{A}\,,\mathcal{B}$ are stably isomorphic (\cite{Elest}). If we define $\cl C=\overline{[M^{\star}\,M]},\;\;\mathcal{D}=\overline{[M\,M^{\star}]}$, then the spaces $$\cl A_0=\overline{\alpha(\cl A)+\cl C},\;\;\; \cl B_0=\overline{\beta(\cl B)+\cl D}$$ 
are operator algebras with contractive approximate identities, even if $\cl A, \cl B$ do not have, and they are also $\Delta$-equivalent since $$\cl A_0=\overline{[M^{\star}\,\mathcal{B}_0\,M]},\;\;\;\mathcal{B}_0=\overline{[M\,\mathcal{A
}_0\,M^{\star]}}.$$ Also observe that
\begin{equation}\label{0000}
\cl A_0=\overline{\cl A_0\cl C}=\overline{\cl C\cl A_0},\;\;\cl B_0=
\overline{\cl B_0\cl D}=\overline{\cl D\cl B_0}
\end{equation}
and that $\alpha(\cl A)$ (resp. $\beta(\cl B)$) is an ideal of $\cl A_0$ (resp. $\cl B_0$).

Generally, if $\cl A_0$ is an operator algebra and $\cl C\subseteq \cl A_0$ 
is a $C^\star$-algebra satisfying relation (\ref{0000}), we call $(\cl A_0,\cl C)$ a $\Delta$-pair. Furthermore, if $\cl C$ has a $\sigma$-unit, we call $(\cl A_0 ,\cl C)$ a $\sigma\Delta$-pair.

In section 2, we characterise the $\Delta$-equivalence and stable isomorphism of $\Delta$-pairs 
under the notion of equivalence of categories of their left operator modules. In section 3, using the above theory, we characterise the $\Delta$-equivalence and stable isomorphism of the operator spaces $\cl X$ and $\cl Y$ through the equivalence of the categories of left operator modules of operator algebras $\cl A_{\cl X}, \cl A_{\cl Y}$, 
on which $\cl X$ and $\cl Y$ naturally embed completely isometrically. If $\cl X$ and $\cl Y$ are unital operator spaces, we get stronger results using the algebras $\Omega_{\cl X}, \Omega_{\cl Y}$ generated by $\cl X, \cl Y$ and the diagonals of their multiplier algebras (see section 4).

If $\cl X$ is an operator space, then $\cl K\otimes \cl X$ is completely isometrically isomorphic with the space $K_\infty(\cl X)$, which is the norm closure of the finitely supported matrices in $M_\infty(\cl X).$ Here, $M_\infty(\cl X)$ is the space of $\infty\times\infty$ matrices, which define bounded operators. Also, by $\cl X\otimes^h\cl Y$, we denote 
the Haagerup tensor product of the operator spaces $\cl X$ and $\cl Y.$ If $\cl A$ is an operator algebra, $\cl X$ is a right $\cl A$-module and $\cl Y$ is a left $\cl A$-module, we denote by $\cl X\otimes^h_{\cl A}\cl Y$ the balanced Haagerup tensor product of $\cl X$ and $\cl Y$ over $\cl A$ (\cite{BMP00}).

For further details about operator spaces, operator algebras, Morita theory and category theory, we refer the reader to \cite{Bas62, BleLeM04, EffRua00, Paul, Pis03, IR}.

\section{$\Delta$-Morita equivalence of operator algebras }

\begin{definition}

Let $\mathcal{A}\subseteq \mathbb{B}(H)\,,\mathcal{B}\subseteq \mathbb{B}(K)$ be operator algebras. We call them TRO-equivalent (resp. $\sigma$-TRO equivalent) if there exists a TRO (resp. $\sigma$-TRO) $M\subseteq \mathbb{B}(H,K)$ such that $$\mathcal{A}=\overline{[M^{\star}\,\mathcal{B}\,M]}\,\,,\mathcal{B}=\overline{[M\,\mathcal{A}\,M^{\star}]}.$$
We write $\mathcal{A}\sim_{TRO} \mathcal{B}$, resp. $\mathcal{A}\sim_{\sigma TRO}\mathcal{B}.$

\end{definition}

\begin{definition}

Let $\mathcal{A}\,,\mathcal{B}$ be operator algebras. We call them $\Delta$-equivalent (resp. $\sigma\Delta$- equivalent) if there exist completely isometric homomorphisms $a:\mathcal{A}\to \mathbb{B}(H)$ and $\beta:\mathcal{B}\to \mathbb{B}(K)$ such that $a(\mathcal{A})\sim_{TRO}\beta(\mathcal{B})$ (resp. $a(\mathcal{A})\sim_{\sigma TRO}\beta(\mathcal{B}).$) We write $\mathcal{A}\sim_{\Delta}\mathcal{B}$ (resp. $\mathcal{A}\sim_{\sigma\,\Delta}\mathcal{B}$)

\end{definition}

\begin{definition}

Let $\mathcal{A}$ be an operator algebra and $C$ be a $C^{\star}$-algebra such that $C\subseteq \mathcal{A}$. If $\mathcal{A}=\overline{[\mathcal{A}\,C]}=\overline{[C\,\mathcal{A}]}$, we call the pair $(\mathcal{A},C)$ a $\,\,\Delta$-pair. If $C$ has a $\sigma$-unit, we call $(\mathcal{A},C)$ a $\sigma\Delta$-pair.

\end{definition}

If $\mathcal{A}$ is an operator algebra, then $\;_{\mathcal{A}}OMOD$ is the category with objects the essential left $\mathcal{A}$-operator modules, namely operator spaces $U$ such that there exists a completely contractive bilinear map $\theta:\mathcal{A}\times U\to U$ such that $U=\overline{[\mathcal{A}\,U]}$, where $\mathcal{A}\,U=\left\{\theta(a,x)\in U\,,a\in\mathcal{A}\,,x\in U\right\}$. For our convenience, we write $a\,x$ instead of $\theta(a,x)$. If $U_1\,,U_2\in\;_{\mathcal{A}}OMOD$ is the space of homomorphisms between $U_1$ and $U_2$ is the space of completely bounded maps, which are left operator maps over $\mathcal{A}$, and we denote this space by $\;_{\mathcal{A}}CB(U_1,U_2)$. Observe that if $(\mathcal{A},C)$ is a $\Delta$-pair, then $\;_{\mathcal{A}}OMOD$ is a subcategory of $\:_{C}OMOD$.

A functor $\cl F: \;_{\cl A}OMOD\rightarrow \;_{\cl B}OMOD$ is called completely contractive if for 
all $U_1, U_2 \in \;_{\mathcal{A}}OMOD$ the map $$\cl F: \;_{\mathcal{A}}CB(U_1,U_2)
\rightarrow\;_{\mathcal{B}}CB(\cl F(U_1),\cl F(U_2))$$
is completely contractive.

\begin{definition}

Let $(\mathcal{A},C)\,,(\mathcal{B},D)$ be $\Delta$-pairs. We call them $\Delta$-Morita equivalent if there exist completely contractive functors $\mathcal{F}: \;_{C}OMOD\to \;_{D}OMOD$ and $G:\;_{D}OMOD\to \;_{C}OMOD$ such that $$G\circ \mathcal{F}\cong Id_{\;_{C}OMOD}\,\,,\mathcal{F}\circ G\cong Id_{\;_{D}OMOD}$$ and $$G|_{\;_{\mathcal{B}}OMOD}\circ \mathcal{F}|_{\;_{\mathcal{A}}OMOD}\cong Id_{\;_{\mathcal{A}}OMOD}\,\,,\mathcal{F}|_{\;_{\mathcal{A}}OMOD}\circ G|_{\;_{\mathcal{B}}OMOD}\cong Id_{\;_{\mathcal{B}}OMOD}.$$ Here, $\cong$ is the natural equivalence. 

\end{definition}

If $\cl A, \cl B, \cl C, \cl D$ are operator algebras such that 
$\cl C\subseteq\cl A, \cl D\subseteq\cl B$ and $\cl A\sim_\Delta\cl B, \;\;\cl C\sim_\Delta\cl D$, we say that $\Delta$-equivalence 
is implemented in both cases by the same TRO if there exist completely 
isometric homomorphisms $\alpha: \cl A\rightarrow \alpha(\cl A)\subseteq \bb B(H)\,, \;\beta:\cl B\rightarrow 
\beta(\cl B)\subseteq \bb B(K) $ and a TRO $M\subseteq \bb B(H,K)$ such that $$\alpha(\mathcal A)=\overline{[M^{\star}\,\beta(\mathcal B)\,M]}\,\,,\beta(\mathcal{B})=\overline{[M\,\alpha(\mathcal{A})\,M^{\star}]}$$ and 
$$\alpha(\mathcal C)=\overline{[M^{\star}\,\beta(\mathcal D)\,M]}\,\,,\beta(\mathcal{D})=\overline{[M\,\alpha(\mathcal{C})\,M^{\star}]}.$$

We now prove our main theorem for operator algebras.

\begin{theorem}

\label{main}
Let $(\mathcal{A},C)\,,(\mathcal{B},D)$ be $\Delta$-pairs. The following are equivalent:\\
$i)\,\mathcal{A}\sim_{\Delta}\mathcal{B}\,\,,C\sim_{\Delta}D$, where $\Delta$-equivalence is implemented in both cases by the same TRO.\\
$ii)$ The pairs $(\mathcal{A},C)\,,(\mathcal{B},D)$ are $\Delta$-Morita equivalent.

\end{theorem}

\begin{proof}

We start with the proof of $i)\implies ii).$

Assume that $\mathcal{A}=\overline{[M^{\star}\,\mathcal{B}\,M]}$ and $\mathcal{B}=\overline{[M\,\mathcal{A}\,M^{\star}]}$ and also $C=\overline{[M^{\star}\,D\,M]}\,,D=\overline{[M\,C\,M^{\star}]}$ for the same TRO $M\subseteq \mathbb{B}(H,K).$

Let $U\in\;_{\mathcal{A}}OMOD$ and $E=\overline{[M^{\star}\,M]}.$ We notice that $$\overline{[E\,U]}=\overline{[M^{\star}\,M\,U]}=\overline{[M^{\star}\,M\,\mathcal{A}\,U]}=\overline{[M^{\star}\,M\,M^{\star}\,\mathcal{B}\,M\,U]}\subseteq \overline{[M^{\star}\,\mathcal{B}\,M\,U]}=\overline{[\mathcal{A}\,U]}=U$$ (so $U$ is a left $E$-operator module).

We set $\mathcal{F}(U)=M\otimes_{E}^h U.$ We fix $$v=\sum_{i=1}^r m_i\,a_i\,n_i^{\star}\in [M\,\mathcal{A}\,M^{\star}].$$

We define the bilinear map $$f_{v}:M\times U\to M\otimes^{h}_{E}U\,,f_{v}(\ell,x)=\sum_{i=1}^r m_i\otimes_{E}a_i\,n_i^{\star}\,\ell\,x$$
and then there exists a linear map denoted again by $f_{v}:M\otimes U\to M\otimes^{h}_{E}U$ such that $$f_{v}(\ell\otimes x)=\sum_{i=1}^r m_i\otimes_{E}a_i\,n_i^{\star}\,\ell\,x\,\,,\ell\in M\,,x\in U$$

Let $$u=\sum_{j=1}^k \ell_{j}\otimes x_j\in M\otimes U.$$
Since $M$ is a TRO, there exists a net $m_{\lambda}=(m_{1\,\lambda}^{\star},...,
m_{n\,\lambda}^{\star})^{t}\in \mathbb{M}_{n\,,1}(M^{\star})$ such that $||m_{\lambda}^{\star}||\leq 1\,,\forall\,\lambda\in\Lambda$ and also $m_{\lambda}\,m_{\lambda}^{\star}\,\ell\to \ell\,,\forall\,\ell\in M,$ (see \cite{BleLeM04}).

Let $\epsilon>0.$ We choose $\lambda_{0}\in\Lambda$ such that for every $\lambda\geq \lambda_0$ holds 

$$||f_{v}(u)||-\epsilon=\nor{\sum_{i=1}^r \sum_{j=1}^k m_i\otimes_{E}a_i\,n_i^{\star}\,\ell_{j}\,x_j}-\epsilon\leq \nor{\sum_{i=1}^r \sum_{j=1}^k m_{\lambda}\,m_{\lambda}^{\star}\,m_i\otimes_{E}a_i\,n_i^{\star}\,\ell_{j}\,x_j}$$

Using now the fact that $||y\otimes_{C}b||\leq ||y||\,||b||\,,y\in M_{p,q}(M^{\star})\,,b\in M_{q, s}(U), \;\;p,q,s \in \bb N,$ we get 

\begin{align*}
 ||f_{v}(u)||-\epsilon&\leq \nor{\sum_{i=1}^r \sum_{j=1}^k m_{\lambda}\otimes_{E}m_{\lambda}^{\star}\,m_i\,a_i\,n_i^{\star}\,\ell_{j}\,x_j}= \nor{m_{\lambda}\otimes_{E}\sum_{i=1}^r m_i\,a_i\,n_i^{\star}\, \sum_{j=1}^k m_{\lambda}^{\star}\,\ell_{j}\,x_j}\\&\leq ||m_{\lambda}||\,\nor{\sum_{i=1}^r m_i\,a_i\,n_i^{\star}}\,\nor{\sum_{j=1}^k m_{\lambda}^{\star}\,\ell_{j}\,x_j}\\&\leq ||v||\,\nor{\sum_{j=1}^k m_{\lambda}^{\star}\,\ell_{j}\otimes x_j}_{h}\leq ||v||\,||(m_{\lambda}^{\star}\,\ell_{1},...,m_{\lambda}^{\star}\,\ell_{k})||\,\nor{\begin{pmatrix}x_1\\
 ...\\
 x_k\end{pmatrix}}\\&\leq ||v||\,||(\ell_{1},...,\ell_{k})||\,\nor{\begin{pmatrix}x_1\\
 ...\\
 x_k\end{pmatrix}}
\end{align*}

We have shown that the above procedure is independent of $\lambda$, so if $\epsilon\to 0^{+}$, and by taking infimum over all representations of $u$, we get $||f_{v}(u)||_h\leq ||v||_h\,||u||_{h}.$ Therefore, $f_{v}$ is continuous and contractive, since $||f_{v}||\leq ||v||_h.$ Let $n\in\mathbb{N}$ and the corresponding map $$(f_{v})_{n}:\mathbb{M}_{n}(M\otimes U)\to \mathbb{M}_{n}(M\otimes_{E}^h U).$$ We have to prove that $(f_{v})_{n}$ is contractive, that is, $f_{v}$ is completely contractive with respect to the Haagerup norm. This statement is true since $$\mathbb{M}_{n}(M)\otimes_{\mathbb{M}_{n}(E)}^h \mathbb{M}_{n}(U)\cong \mathbb{M}_{n}(M\otimes_{E}^h U)\,,n\in\mathbb{N}.$$

For more details check \cite{BleLeM04}.

Furthermore, for every $\ell\in M\,,z^{\star}\,w\in M^{\star}\,M\,,x\in U$ holds \begin{align*}
 f_{v}(\ell\,z^{\star}\,w\otimes x)&=\sum_{i=1}^r m_i\otimes_{E}a_i\,n_i^{\star}\,\ell\,z^{\star}\,w\,x\\&=\sum_{i=1}^r m_i\otimes_{E}a_i\,n_i^{\star}\,\ell\,(z^{\star}\,w\,x)\\&=f_{v}(\ell\otimes z^{\star}\,w\,x)
\end{align*}

Since $f_{v}$ is continuous and linear and $E=\overline{[M^{\star}\,M]}$, we get $f_{v}(\ell\,e\otimes x)=f_{v}(\ell\otimes e\,x)$ for every $\ell\in M\,,e\in E\,,x\in U.$ Therefore, $f_{v}$ extends to a linear and completely contractive map $$\hat{f_{v}}:M\otimes_{E}^h U\to M\otimes_{E}^h U$$ with the property $$\hat{f_{v}}(\ell\otimes_{E}x)=\sum_{i=1}^r m_i\otimes_{E}a_i\,n_i^{\star}\,\ell\,x\,,\ell\in M\,,x\in U$$

So, we have the map $\hat{f}:[M\,\mathcal{A}\,M^{\star}]\to CB(\mathcal{F}(U))\,\,,v\mapsto \hat{f}_{v}$, which is completely contractive and therefore extends to a completely contractive map denoted again by $\hat{f}:\mathcal{B}\to CB(\mathcal{F}(U))$, where $CB(\mathcal{F}(U))$ is the space of all linear and completely bounded maps of $\mathcal{F}(U)$ to itself. The algebra $\mathcal{B}$ acts to $\mathcal{F}(U)$ via the map $$\hat{\theta}:\mathcal{B}\times \mathcal{F}(U)\to \mathcal{F}(U)\,,\hat{\theta}(b,y)=\hat{f}(b)(y),$$  such that $\overline{[\mathcal{B}\,\mathcal{F}(U)]}=\mathcal{F}(U)$ and thus $\mathcal{F}(U)=M\otimes_{E}^h U\in\;_{\mathcal{B}}OMOD.$

Therefore, we have a correspondence between the objects 
$$\mathcal{F}:\;_{\mathcal{A}}OMOD\to \;_{\mathcal{B}}OMOD, U\mapsto \mathcal{F}(U)=
M\otimes_{E}^h U.$$

Let $\,\,U_1\,,U_2\in\;_{\mathcal{A}}OMOD.$ We fix $f\in\;_{\mathcal{A}}CB(U_1,U_2)$ and we define the map $$\mathcal{F}(f):M\times U_1\to M\otimes_{E}^h U_2=\mathcal{F}(U_2)\,\,,\mathcal{F}(f)(\ell,x):=\ell\otimes_{E} f(x)$$

The map $\mathcal{F}(f)$ is linear, completely contractive and $E$-balanced, so we denote again by $\mathcal{F}(f)$ the linear and completely contractive map $$\mathcal{F}(f):M\otimes_{E}^h U_1=\mathcal{F}(U_1)\to M\otimes_{E}^h U_2=\mathcal{F}(U_2)$$ with the property $$\mathcal{F}(f)(\ell\otimes_{E}x)=\ell\otimes_{E}f(x)\,,\ell\in M\,,x\in U$$

Furthermore, \begin{align*}
 &\mathcal{F}(f)(m\,a\,n^{\star}\cdot \ell\otimes_{E} x)=\mathcal{F}(f)(m\otimes_{E}a\,n^{\star}\,\ell\,x)=m\otimes_Ef(an^*lx)=\\&m\otimes_{E}a\,n^{\star}\,\ell\,f(x)=m\,a\,n^{\star}\cdot \mathcal{F}(f)(\ell\otimes_{E}x)\,,m\,,n\,,\ell\in M\,,x\in U_1\,,a\in\mathcal{A}
\end{align*}and since 

$\mathcal{B}=\overline{[M\,\mathcal{A}\,M^{\star}]}$, we have 

 $$\mathcal{F}(f)(b\cdot y) =b\cdot \mathcal{F}(f)(y)\,,b\in\mathcal{B}\,,y\in M\otimes_{E}^h U_1.$$

We proved that $\mathcal{F}(f)\in\;_{\mathcal{B}}CB(\mathcal{F}(U_1),\mathcal{F}(U_2))$

Therefore, we have a completely contractive map $$\mathcal{F}:\;_{\mathcal{A}}CB(U_1,U_2)\to \;_{\mathcal{B}}CB(\mathcal{F}(U_1),\mathcal{F}(U_2))\,,f\mapsto \mathcal{F}(f)$$

Similarly, we have a functor $G:\;_{\mathcal{B}}OMOD\to \;_{\mathcal{A}}OMOD$ defined as $$G(V)=M^{\star}\otimes_{E'}^h V\,,V\in\;_\mathcal{B}OMOD,$$ where $E'=\overline{[M\,M^{\star}]}$
and the corresponding functor 

$$G:\;_{\mathcal{B}}CB(V_1,V_2)\to \;_{\mathcal{A}}CB(G(V_1),G(V_2))$$

for every $V_1,V_2\in\;_{\mathcal{B}}OMOD$.

 We are going to prove that 
 $G$ is the natural inverse of $\mathcal{F}.$ If 
 $U\in\;_{\mathcal{A}}OMOD$, we have that

\begin{align*}
 (G\,\mathcal{F})(U)&=G(M\otimes_{E}^h U)\\&=M^{\star}\otimes_{E'}^h (M\otimes_{E}^h U)\cong (M^{\star}\otimes_{E'}^h M)\otimes_{E}^h U\cong E\otimes_{E}^h U\\& \cong U=Id_{\mathcal{C}^{OMOD}}(U)
\end{align*}

(Similarly, $(\mathcal{F}\,G)(V)\cong V\,,\forall\,V\in\;_{\mathcal{B}}OMOD$).

We note that if $U\in\;_{\mathcal{A}}OMOD$, then there exists an isometry $$f_{U}:(G\,F)(U)=M^{\star}\otimes_{E'}^h (M\otimes_{E}^h U)\to U=Id_{\;_{\mathcal{A}}OMOD}(U)$$ such that $$f_{U}(m^{\star}\otimes_{E'}(\ell\otimes_{E}x))=m^{\star}\,\ell\,x\,,m\,,\ell\in M\,,x\in U.$$

We have to prove that for every $U_1\,,U_2\in\;_{\mathcal{A}}OMOD\,\,,f\in\;_{\mathcal{A}}CB(U_1,U_2)$, the following diagram

\begin{center}

\begin{tikzcd}[column sep=huge,row sep=huge]

(G\,\mathcal{F})(U_1) \arrow[d,"(G\,\mathcal{F})(f)" ] \arrow[r, "f_{U_1}"] \arrow[d, black]

& U_1 \arrow[d, "f" black] \\

(G\,\mathcal{F})(U_2) \arrow[r, black, "f_{U_2}" black]

& U_2

\end{tikzcd}

\end{center}

is commutative, or equivalently, the following diagram is commutative 

\begin{center}

\begin{tikzcd}[column sep=huge,row sep=huge]

M^{\star}\otimes_{E'}^h (M\otimes_{E}^h U_1) \arrow[d,"G(\mathcal{F}(f))" ] \arrow[r, "f_{U_1}"] \arrow[d, black]

& U_1 \arrow[d, "f" black] \\

M^{\star}\otimes_{E'}^h (M\otimes_{E}^h U_2) \arrow[r, black, "f_{U_2}" black]

& U_2

\end{tikzcd}

\end{center}

So, we have to prove that $f\circ f_{U_1}=f_{U_2}\circ G(\mathcal{F}(f))$.

Indeed,

\begin{align*}
&(f_{U_2}\circ G(\mathcal{F}(f))(m^{\star}\otimes_{E'}(\ell\otimes_{E}x))=f_{U_2}(G(\mathcal{F}(f))(m^{\star}\otimes_{E'}(\ell\otimes_{E}x))=\\&f_{U_2}(m^{\star}\otimes_{E'}\mathcal{F}(f)(\ell\otimes_{E}x))=f_{U_2}(m^{\star}\otimes_{E'}\ell\otimes_{E}f(x))=m^{\star}\,\ell\,f(x)=f(m^{\star}\,\ell\,x)
\end{align*}

and on the other hand $$(f\circ f_{U_1})(m^{\star}\otimes_{E'}(\ell\otimes_{E}x))=f(f_{U_1}(m^{\star}\otimes_{E'}(\ell\otimes_{E}x))=f(m^{\star}\,\ell\,x)$$
for every $m\,,\ell\in M\,,x\in U_1.$

The functor $\mathcal{F}$ extends to a functor $\mathcal{F}^{\delta}$ to the category $\;_{\cl C}OMOD$ in the same sense that is $\mathcal{F}^{\delta}(U)=M^{\star}\otimes_{C}^h U\,,U\in \;_{\cl C}^{OMOD}$ and $\mathcal{F}^{\delta}|_{\;_\mathcal{A}OMOD}=\mathcal{F}$ (similarly for $G^{\delta}$). 
In conclusion, we have proved that the pairs $(\mathcal{A},C)\,,(\mathcal{B},D)$ are $\Delta$-Morita equivalent.\\

We are now going to complete the remaining proof of $ii)\implies i)$. Suppose that the pairs $(\mathcal{A},\cl C)\,,(\mathcal{B},\cl D)$ are $\Delta$-Morita equivalent. We fix an equivalence functor $\mathcal{F}:\;_{C}OMOD\to \;_{D}OMOD$ with inverse $G:\;_{D}OMOD\to \;_{\cl C}OMOD$ such that $$\mathcal{F}(\;_{\mathcal{A}}OMOD)=\;_{\mathcal{B}}OMOD\,\,,G(\;_{\mathcal{B}}OMOD)=\;_{\mathcal{A}}OMOD$$

Let $\mathcal{F}(\cl C)=\cl Y_0\,\,,G(\cl D)=\cl X_0$. By \cite{Ble} we have that $\cl Y_0$ is a TRO and $\cl X_0\cong \cl Y_0^{\star}$. Also, $\cl C\cong \cl X_0\otimes_{\cl D}^h \cl Y_0\,\,,\cl D\cong \cl Y_0\otimes_{C}^h \cl X_0.$ We also assume that $\mathcal{F}(\mathcal{A})=\cl Y\,,G(\mathcal{B})=\cl X$, and by \cite{Ble01-2} we get $$\mathcal{A}\cong \cl X\otimes_{\mathcal{B}}^h \cl Y\,\,,\mathcal{B}\cong \cl Y\otimes_{\mathcal{A}}^h \cl X.$$ 

From both the above papers, we have that $$\mathcal{F}(U)\cong \cl Y\otimes_{\mathcal{A}}^h U\,,\forall\,U\in\;_{\mathcal{A}}OMOD\,\,,\mathcal{F}(U)\cong \cl Y_0\otimes_{C}^h U\,,\forall\,U\in \;_{C}OMOD$$ and then we get $$\cl Y\otimes_{\mathcal{A}}^h U\cong \cl Y_0\otimes_{C}^h U\,,\forall\,U\in\;_{\mathcal{A}}OMOD.$$ Similarly, $\cl X\otimes_{\mathcal{B}}^h V\cong \cl X_0\otimes_{D}^h V\,,\forall\,V\in\;_{\mathcal{B}}OMOD.$

Now, we have that

\begin{align*}
\cl X_0\otimes_{D}^h \mathcal{B}\otimes_{D}^h \cl X_0^{\star}&\cong \cl X_0\otimes_{D}^h (\cl Y\otimes_{\mathcal{A}}^h \cl X)\otimes_{D}^h \cl Y_0\\&\cong (\cl X_0\otimes_{D}^h \cl Y)\otimes_{\mathcal{A}}^h (\cl X\otimes_{D}^h \cl Y_0)\\&\cong (\cl X\otimes_{\mathcal{B}}^h \cl Y)\otimes_{\mathcal{A}}^h (\cl X\otimes_{\mathcal{B}}^h \cl Y)\\&\cong \mathcal{A}\otimes_{\mathcal{A}}^h \mathcal{A}\\&\cong \mathcal{A}
\end{align*}

Similarly, $\cl X_0^{\star}\otimes_{C}^h \mathcal{A}\otimes_{C}^h \cl X_0\cong \mathcal{B}$. The following lemma implies that $\mathcal{A}\sim_{\Delta}\mathcal{B}$ and 
$\mathcal{C}\sim_{\Delta}\mathcal{D},$ where $\Delta$-equivalence is 
implemented in both cases by the same TRO. The proof of Theorem 1.1 is complete.

\end{proof}

\begin{lemma}

Suppose that $\mathcal{A}\,\,,\mathcal{B}$ are operator algebras and $D\subseteq \mathcal{B}$ be a $C^{\star}$ - algebra such that $\overline{[D\,\mathcal{B}]}=\overline{[\mathcal{B}\,D]}=\mathcal{B}$. Let $M\subseteq \mathbb{B}(K,H)$ be a TRO such that $\overline{[M^{\star}\,M]}\cong D$ (as $C^{\star}$ algebras) and assume that $\mathcal{A}\cong M\otimes_{D}^h \mathcal{B}\otimes_{D}^h M^{\star}.$ Then $\mathcal{A}\sim_{\Delta} \mathcal{B}$.

\end{lemma}

\begin{proof}

We fix a completely isometric homomorphism $\beta:\mathcal{B}\to \mathbb{B}(K)$, and we have that $\beta|_{D}$ is also a $\star$-homomorphism. We can consider that the operator space $M\otimes_{D}^h K$ is a Hilbert space with the inner product given by $$\langle{n\otimes_{D}x,\ell\otimes_{D}y\rangle}:=\langle{\beta(\phi(\ell^{\star}\,n))(x),y\rangle}\,,n\,,\ell\in M\,,x\,,y\in K,$$ where $\phi:\overline{[M^{\star}\,M]}\to D$ is an isometric $\star$-isomorphism (\cite{BMP00}). Instead of $\phi(\ell^{\star}\,n)$, we may write $$\langle{n\otimes_{D}x,\ell\otimes_{D}y\rangle}:=\langle{\beta(\ell^{\star}\,n)(x),y\rangle}\,,n\,,\ell\in M\,,x\,,y\in K,$$
and for the action of $D$ on $K$, we denote $\ell^{\star}\,n\,x$ instead of $\phi(\ell^{\star}\,n)\,x$ where $\ell\,,n\in M\,,x\in K.$

For each $m\in M$, we define $r_m:K\to M\otimes_{D}^h K$ by $r_m(x)=m\otimes_{D}x$. 
Obviously, $r_m$ is a linear map and $r_m\in\mathbb{B}(K,M\otimes_{D}^h K)$.

We can easily see that 
$r_{m_1}\,r_{m_2}^{\star}\,r_{m_3}=r_{m_1\,m_2^{\star}\,m_3}\in r(M),$ therefore, $r(M)$ is a TRO. Also, with similar arguments, we have that $\beta(D)=\overline{[r(M)^{\star}\,r(M)]}$ (since $r(M)^{\star}\,r(M)=\beta(\phi(M^{\star}\,M))$). We also claim that $r$ is completely isometric. By Lemma 8.3.2 (Harris-Kaup) of \cite{BleLeM04}, it is sufficient to prove that $r$ is one-to-one. Indeed, for every $m\in M$ holds $$(r_m^{\star}\,r_m)(x)=r_m^{\star}(m\otimes_{D}x)=\beta(m^{\star}\,m)(x)\,,\forall\,x\in K$$ so $||r_m||^2=||r_m^{\star}\,r_m||=||\beta(m^{\star}\,m)||=||m^{\star}\,m||=||m||^2$, which means that $r$ is isometric and also one-to-one. Therefore, $M\cong r(M)$, and using Lemma 5.4 in \cite{Elekak}, 
we get 
 $$M\otimes_{D}^h \mathcal{B}\otimes_{D}^h M^{\star}\cong M\otimes_{D}^h \overline{[\beta(\mathcal{B})
 \,r(M)^{\star}]} \cong \overline{[r(M)\,\beta(\mathcal{B})\,r(M)^{\star}]}.$$ Therefore, there exists a completely isometric map such that $a(\mathcal{A})=\overline{[r(M)\,\beta(\mathcal{B})\,r(M)^{\star}]}\,(4)$, so
\begin{align*}
 \overline{[r(M)^{\star}\,a(\mathcal{A})\,r(M)]}&=\overline{[r(M)^{\star}\,r(M)\,\beta(\mathcal{B})\,r(M)^{\star}\,r(M)]}\\&=\overline{[\beta(D)\,\beta(\mathcal{B})\,\beta(D)]}\\&=\overline{[\beta(D\,\mathcal{B}\,D)]}\\&=\beta(\mathcal{B})\,(5)
\end{align*}

By $(4)\,,(5)$, we get $a(\mathcal{A})\sim_{TRO}\beta(\mathcal{B})\implies \mathcal{A}\sim_{\Delta}\mathcal{B}.$

\end{proof}

\begin{corollary}

The relation $\sim_{\Delta}$ is an equivalence relation for $\Delta$-pairs.

\end{corollary}

\begin{remark}\em{ We consider that the $\Delta$-pairs $(\cl A, \cl C), (\cl B, \cl D)$ are equivalent 
in the sense of Theorem \ref{main}, and $\cl F$ is the functor defined in its proof.
For every $U_1\,,U_2\in\;_{\mathcal{A}}OMOD$, the map $\mathcal{F}:\;_{\mathcal{A}}CB(U_1,U_2)\to\;_{\mathcal{B}}CB(\mathcal{F}(U_1),\mathcal{F}(U_2))$ is a complete isometry.}
\end{remark}

\begin{proof}

For every $g\in\;_{\mathcal{B}}CB(\mathcal{F}(U_1),\mathcal{F}(U_2))$, we define $$\theta=f_{U_2}\circ g\circ f_{U_1}^{-1}\in\;_{\mathcal{A}}CB(U_1,U_2).$$

So, for every $f\in\;_{\mathcal{A}}CB(U_1,U_2)$, we have that $\mathcal{F}(f)\in\;_{\mathcal{B}}CB(\mathcal{F}(U_1),\mathcal{F}(U_2))$ and

\begin{align*}
 (\theta\circ G)(\mathcal{F}(f))&=\theta(G(\mathcal{F}(f))\\&=f_{U_2}\circ G\,\mathcal{F}(f)\circ f_{U_1}^{-1}\\&=f\circ f_{U_1}\circ f_{U_1}^{-1}\\&=f
\end{align*}

Since $\theta\circ G$ is completely contractive, we get $$||f||_{cb}=||(\theta\circ G)(\mathcal{F}(f)||_{cb}\leq ||\mathcal{F}(f)||_{cb}.$$

\end{proof}

\begin{theorem}

Let $(\mathcal{A},C)\,,(\mathcal{B},D)$ be $\sigma\,\Delta$-pairs. The following are equivalent:\\
$i)\,\mathcal{A}\sim_{\sigma\,\Delta}\mathcal{B}\,\,,C\sim_{\sigma\,\Delta}D$, where $\sigma\,\Delta$-equivalence is implemented in both cases by the same $\sigma$-TRO.\\
$ii)$ The pairs $(\mathcal{A},C)\,,(\mathcal{B},D)$ are $\Delta$-Morita equivalent.\\
$iii)$ There exists a completely isometric isomorphism $\phi:\mathcal{A}\otimes \cl K\to \mathcal{B}\otimes \cl K$ such that $\phi(C\otimes \cl K)=D\otimes \cl K$, where $\cl K$ is the algebra of compact operators of $\ell^2(\mathbb{N}).$

\end{theorem}

\begin{proof}

$i)\iff ii)$ It is obvious according to the previous Theorem \ref{main}\\

$i)\implies iii)$. We may consider $\mathcal{A}=\overline{[M^{\star}\,\mathcal{B}\,M]}\,\,,\mathcal{B}=\overline{[M\,\mathcal{A}\,M^{\star}]}$ and also $C=\overline{[M^{\star}\,D\,M]}\,\,,D=\overline{[M\,C\,M^{\star}]}.$ Since $C\,,D$ have a $\sigma$-unit by Lemma 3.4 of \cite{Ele14}, $M$ is a $\sigma$-TRO. 
By Theorem 3.2 in the same article, there exists a completely isometric onto map $\phi:\mathcal{A}\otimes \cl K\to \mathcal{B}\otimes \cl K$ such that $\phi(C\otimes \cl K)=D\otimes \cl K.$\\

$iii)\implies i)$ We have that $(\mathcal{A},C)\sim_{\Delta}(\mathcal{A}\otimes \cl K,C\otimes \cl K)$, so $(\mathcal{A},C)\sim_{\Delta}(\mathcal{B}\otimes \cl K,D\otimes \cl K)$, but also, $(\mathcal{B},D)\sim_{\Delta} (\mathcal{B}\otimes \cl K,D\otimes \cl K)$. Since $\sim_{\Delta}$ is an equivalence relation for $\Delta$-pairs, we get $(\mathcal{A},C)\sim_{\Delta} (\mathcal{B},D).$

\end{proof}

In the rest of this section, we consider that the $\Delta$-pairs $(\cl A, \cl C), (\cl B, \cl D)$ are equivalent in the sense of Theorem \ref{main}, and $\cl F$ is the functor defined in its proof.

We consider the subcategory of representations of $\mathcal{A}$ denoted by $\;_{\mathcal{A}}HMOD$. If $H'\in\;_{\mathcal{A}}HMOD$, there exists a completely contractive morphism $\pi:\mathcal{A}\to \mathbb{B}(H')$ such that $\overline{\pi(\mathcal{A})(H')}=\overline{[\pi(a)(h)\in H': a\in\mathcal{A}\,,h\in H']}$. The space $\mathcal{F}(H')=M\otimes_{E}^h H'$ is also a Hilbert space. Its inner product is given by $$\langle{m\otimes_{E}\xi,\ell\otimes_{E}w\rangle}:=\langle{\pi(\ell^{\star}\,m)(\xi),w\rangle}_{H'}\,,m\,,\ell\in M\,,\xi\,,w\in H'$$

(for more details check \cite{Rie74, Rie82, BMP00}).

Also, the map $\mathcal{F}(\pi):\mathcal{A}\to \mathbb{B}(\mathcal{F}(H'))$ given by $$\mathcal{F}(\pi)(m\,b\,n^{\star})(\ell\otimes_{E}\xi)=m\otimes_{E}\pi(b\,n\,\ell^{\star})(\xi)$$ is completely contractive. 

We are going to prove that the functor $\mathcal{F}$ maintains the complete isometric representations. So, let $\pi:\mathcal{A}\to \mathbb{B}(H')$ be a homomorphism and also a complete isometry. We set $\rho=\mathcal{F}(\pi):\mathcal{B}\to \mathbb{B}(\mathcal{F}(H'))$, where $\mathcal{F}(H')=M\otimes_{E}^h H'$ and $$\rho(m\,a\,n^{\star})(\ell\otimes_{E}h)=m\otimes_{E}\pi(a\,n\,\ell^{\star})(h)\,,m\,,n\,,\ell\in M\,,a\in\mathcal{A}\,,h\in H'$$

We define the unitary operator $$U:G\,\mathcal{F}(H')\to H'\,,U(k^{\star}\otimes_{E'}(n\otimes_{E}h)):=\pi(k^{\star}\,n)(h),$$ and we consider $\phi=G(\rho):\mathcal{A}\to \mathbb{B}(G\,\mathcal{F}(H'))$ given by $$\phi(m^{\star}\,b\,n)(\ell^{\star}\otimes_{E'}x)=m^{\star}\otimes_{E'}\rho(b\,n\,\ell^{\star})(x)\,,m\,,n\,,\ell\in M\,,b\in\mathcal{B}\,,x\in \mathcal{F}(H')$$

\begin{lemma}

It holds that $U\,\phi(a)\,U^{\star}=\pi(a)\,,\forall\,a\in\mathcal{A}$.

\end{lemma}

\begin{proof}

For every $m\,,k\,,s\,,n\,,t\,,\ell\in M\,\,,a\in\mathcal{A}\,,h\in H'$, we have that 

\begin{align*}
 \phi(m^{\star}\,k\,a\,s^{\star}\,n)(\ell^{\star}\otimes_{E'}(t\otimes_{E}h))&=m^{\star}\otimes_{E'}\rho(k\,a\,s^{\star}\,n\,\ell^{\star})(t\otimes_{E}h)\\&=m^{\star}\otimes_{E'}k\otimes_{E}\pi(a\,s^{\star}\,n\,\ell^{\star}\,t)(h)
\end{align*}

Therefore,

\begin{align*}
 U(\phi(m^{\star}\,k\,a\,s^{\star}\,n)(\ell^{\star}\otimes_{E'}(t\otimes_{E} h)))&=U(m^{\star}\otimes_{E'}k\otimes_{E}\pi(a\,s^{\star}\,n\,\ell^{\star}\,t)(h))\\&=\pi(m^{\star}\,k)(\pi(a\,s^{\star}\,n\,\ell^{\star}\,t)(h))\\&=\pi(m^{\star}\,k\,a\,s^{\star}\,n)\,\pi(\ell^{\star}\,t)(h)\\&=\pi(m^{\star}\,k\,a\,s^{\star}\,n)\,U(\ell^{\star}\otimes_{E'}(t\otimes_{E}h))
\end{align*}

So, $U(\phi(m^{\star}\,k\,a\,s^{\star}\,n))=\pi(m^{\star}\,k\,a\,s^{\star}\,n)\,U$, but since $\mathcal{A}=\overline{[M^{\star}\,M\,\mathcal{A}\,M^{\star}\,M]}$, we get that $U\,\phi(a)\,U^{\star}=\pi(a).$

\end{proof}

We conclude that since $\pi$ is an isometry and $U\,\phi(a)\,U^{\star}=\pi(a)\,,a\in\mathcal{A}$, where $U$ is unitary, $\phi$ is also an isometry.
Observe now that if $(m_i)_{i\in I}$ is a net of $\mathbb{M}_{n_i,1}(M^{\star})$ such that $||m_i||\leq 1\,\forall\,i\in I$ and also $m_i\,m_i^{\star}\,m\to m\,,\forall\,m\in M$, then we have that 

\begin{align*}
 \phi(m_i^{\star}\,b\,m_i)(\ell^{\star}\otimes_{E'}x)&=m_i^{\star}\otimes_{E'}\rho(b\,m_i\,\ell^{\star})(x)\\&=m_i^{\star}\otimes_{E'}\rho(b)\,V(m_i\otimes_{E}(\ell^{\star}\otimes_{E'}x))
\end{align*}
for every $b\in\mathcal{B}\,,x\in\mathcal{F}(H')\,,\ell\in M$ where $V$ is the unitary operator $$V:M\otimes_{E}^h (M^{\star}\otimes_{E'}^h K')\to K'\,\,,K'=\mathcal{F}(H').$$

Therefore, for all $w\in \overline{M^{\star}\otimes K'}$ holds $$\phi(m_i^{\star}\,b\,m_i)(w)=m_i\otimes_{E}\rho(b)\,V(m_i\otimes_{E}w).$$ So, $$||\phi(m_i^{\star}\,b\,m_i)(w)||\leq ||m_i||\,||\rho(b)||\,||m_i^{\star}||\,||w||\leq ||\rho(b)||\,||w||,$$ which means that $||\phi(m_i^{\star}\,b\,m_i)||\leq ||\rho(b)||$, but $\phi$ is an isometry, and we conclude that $||m_i^{\star}\,b\,m_i||\leq ||\rho(b)||\,,\forall\,b\in\mathcal{B}\,(3).$

Since $\lim_{i}m_i\,m_i^{\star}\,b\,m_i\,m_i^{\star}=b$, we have that $$\sup_{i\in I}||m_i\,m_i^{\star}\,b\,\,m_im_i^{\star}||=||b||.$$ On the other hand, $||m_i\,m_i^{\star}\,b\,m_i\,m_i^{\star}||\leq ||m_i^{\star}\,b\,m_i||\leq ||b||$, therefore, $$\sup_{i\in I}||m_i^{\star}\,b\,m_i||=||b||\,,\forall\,b\in\mathcal{B}.$$

We conclude from $(3)$ that $||b||\leq ||\rho(b)||\,,\forall\,b\in\mathcal{B}$, but also that $\rho$ is completely contractive and therefore $||\rho(b)||=||b||\,,\forall\,b\in\mathcal{B}$, so $\rho$ is an isometry. Similarly, $\rho$ is a complete isometry.
Using the above facts, we can prove that $\cl F$ restricts to an equivalence functor from $\;_{\cl A}HMOD$ to $\;_{\cl B}HMOD.$ This functor maps completely isometric representations to completely isometric representations.

\section{$\Delta$-Morita equivalence of operator spaces}

\begin{definition}

Let $\cl X\subseteq \mathbb{B}(H_1,H_2)\,,\cl Y\subseteq \mathbb{B}(K_1,K_2)$ be operator spaces. We call them TRO-equivalent (resp. $\sigma$-TRO equivalent) if there exist TROs (resp. $\sigma$-TROs) $M_i\subseteq \mathbb{B}(H_i,K_i)\,,i=1,2$ such that $$\cl X=\overline{[M_2^{\star}\,\cl Y\,M_1]}\,\,,\cl Y=\overline{[M_2\,\cl X\,M_1^{\star}]}$$
We write $\cl X\sim_{TRO} \cl Y$, resp. $\cl X\sim_{\sigma TRO}\cl Y$.
\end{definition}

\begin{definition}
Let $\cl X\,,\cl Y$ be operator spaces. We call them $\Delta$-equivalent (resp. $\sigma\,Delta$-equivalent) if there exist completely isometric maps $\phi:\cl X\to \mathbb{B}(H_1,H_2)\,,\psi:\cl Y\to \mathbb{B}(K_1,K_2)$ such that $\phi(\cl X)\sim_{TRO}\psi(\cl Y)$ (resp. $\phi(\cl X)\sim_{\sigma TRO}\psi(\cl Y).$ We write $\cl X\sim_{\Delta} \cl Y$, resp. $\cl X\sim_{\sigma\,\Delta} \cl Y.$ 
\end{definition}

\begin{definition}
Let $\cl X$ be an operator space and $D_1\,,D_2$ be $C^{\star}$-algebras (resp. $\sigma$ unital $C^{\star}$-algebras) such that $$\cl X=\overline{[D_1\,\cl X]}=\overline{[D_2\,\cl X]}$$ Then, the 
space $$\mathcal{A}_{\cl X}=\begin{pmatrix}D_2 & \cl X\\
0 & D_1\end{pmatrix}$$ is an operator algebra, which we call an algebraic $\Delta$-extension of $X$ (resp. $\sigma\,\Delta$-extension of $\cl X$).
\end{definition}

\begin{definition}
Let $\cl X\,,\cl Y$ be operator spaces. We call them $\Delta$-Morita equivalent (resp. $\sigma\,\Delta$-Morita equivalent) if they have algebraic $\Delta$-extensions (resp. $\sigma\,\Delta$-extensions) $\mathcal{A}_{\cl X}\,,\mathcal{A}_{\cl Y}$ such that the $\Delta$-pairs $(\mathcal{A}_{\cl X},\Delta(\mathcal{A}_{\cl X}))\,,(\mathcal{A}_{\cl Y},\Delta(\mathcal{A}_{\cl Y}))$ to be $\Delta$-Morita equivalent.
\end{definition}

\begin{lemma}

\label{lem1} 
The TRO-equivalence (resp. $\sigma$-TRO) of operator spaces is an equivalence relation.

\end{lemma}

\begin{proof}The fact TRO-equivalence is an equivalence relation has been proved in 
\cite{Elekak}. The proof that $\sigma$-TRO-equivalence is an equivalence relation is similar.
\end{proof}

\begin{theorem}

\label{forunitops}

Let $\cl X\,,\cl Y$ be operator spaces. The following are equivalent:\\
$i)\,\,\cl X\sim_{\Delta}\cl Y$\\
$ii)$ $\cl X$ and $\cl Y$ are $\Delta$-Morita equivalent.

\end{theorem}

\begin{proof}
$i)\implies ii)$ We may assume that $\cl X=\overline{[M_2^{\star}\,\cl Y\,M_1]}\,\,,\cl Y=\overline{[M_2\,\cl X\,M_1^{\star}]}$ for TROs $M_1\subseteq \mathbb{B}(H_1,K_1)$ and $M_2\subseteq \mathbb{B}(H_2,K_2).$ If we consider the $C^{\star}$-algebras $$D_1=\overline{[M_1^{\star}\,M_1]}\,,D_2=\overline{[M_2^{\star}\,M_2]}\,,E_1=\overline{[M_1\,M_1^{\star}]}\,,E_2=\overline{[M_2\,M_2^{\star}]},$$  we get $$\cl X=\overline{[D_2\,\cl X]}=\overline{[\cl X\,D_1]}\,\,,\cl Y=\overline{[E_2\,\cl Y]}=\overline{[\cl Y\,E_1]}.$$ So, the operator algebras $$\mathcal{A}_{\cl X}=\begin{pmatrix}\cl D_2 & \cl X\\

0 & \cl D_1\end{pmatrix}\subseteq \mathbb{B}(H_2\oplus H_1)\,\,,\mathcal{A}_{\cl Y}=\begin{pmatrix}E_2 & \cl Y\\

0 & E_1 \end{pmatrix}\subseteq \mathbb{B}(K_2\oplus K_1)$$ are $\Delta$-algebraic extensions of $\cl X\,,\cl Y$, respectively, such that $$\Delta(\mathcal{A}_{\cl X})=\begin{pmatrix}D_2 & 0\\

0 & D_1\end{pmatrix}\,,\Delta(\mathcal{A}_{\cl Y})=\begin{pmatrix}E_2 & 0\\

0 & E_1\end{pmatrix}$$

Clearly $$M=\begin{pmatrix}M_2 & 0\\

0 & M_1\end{pmatrix}\subseteq \mathbb{B}(H_2\oplus H_1,K_2\oplus K_1)$$ is a TRO.

Furthermore, $\overline{[M^{\star}\,\mathcal{A}_{\cl Y}\,M]}=\mathcal{A}_{\cl X}$
and $$\mathcal{A}_{\cl Y}=\overline{[M\,\mathcal{A}_{\cl X}\,M^{\star}]}\,\,,\Delta(\mathcal{A}_{\cl X})=\overline{[M^{\star}\,\Delta(\mathcal{A}_{\cl Y})\,M]}\,\,,\Delta(\mathcal{A}_{\cl Y})=\overline{[M\,\Delta(\mathcal{A}_{\cl X})\,M^{\star}]},$$ 
so the pairs $(\mathcal{A}_{\cl X},\Delta(\mathcal{A}_{\cl X}))\,\,,(\mathcal{A}_{\cl Y},\Delta(\mathcal{A}_{\cl Y}))$ are $\Delta$-Morita equivalent. That is, $\cl X$ and $\cl Y$ are $\Delta$-Morita equivalent.\\

$ii)\implies i)$ Suppose that $\cl X$ and $\cl Y$ are $\Delta$-Morita equivalent. There exist $C^{\star}$-algebras $D_i\,,E_i\,,i=1,2$ such that $\cl X=\overline{[D_2\,\cl X]}=\overline{[\cl X\,D_1]}\,\,,\cl Y=\overline{[E_2\,\cl Y]}=\overline{[\cl Y\,E_1]}$ and the pairs $(\mathcal{A}_{\cl X},\Delta(\mathcal{A}_{\cl X}))\,,(\mathcal{A}_{\cl Y},\Delta(\mathcal{A}_{\cl Y}))$ are $\Delta$-Morita equivalent, where $$\mathcal{A}_{\cl X}=\begin{pmatrix}D_2 & \cl X\\

0 & D_1\end{pmatrix}\,\,,\mathcal{A}_{\cl Y}=\begin{pmatrix}E_2 & \cl Y\\

0 & E_1\end{pmatrix}$$ so $$\Delta(\mathcal{A}_{\cl X})=\begin{pmatrix}D_2 & 0\\

0 & D_1\end{pmatrix}\,\,,\Delta(\mathcal{A}_{\cl Y})=\begin{pmatrix}E_2 & 0\\

0 & E_1\end{pmatrix}$$

Let $N$ be a TRO such that $$\mathcal{A}_{\cl X}=\overline{[N^{\star}\,\mathcal{A}_{\cl Y}\,N]}\,,\mathcal{A}_{\cl Y}=\overline{[N\,\mathcal{A}_{\cl X}\,N^{\star}]}\,,\Delta(\mathcal{A}_{\cl X})=\overline{[N^{\star}\,\Delta(\mathcal{A}_{\cl Y})\,N]}\,,\Delta(\mathcal{A}_{\cl Y})=\overline{[N\,\Delta(\mathcal{A}_{\cl X})\,N^{\star}]}.$$

We define $M=\overline{[\Delta(\mathcal{A}_{\cl Y})\,N]}=\overline{[N\,\Delta(\mathcal{A}_{\cl X})]}$, then $M$ is a TRO since \begin{align*}M\,M^{\star}\,M&=\overline{[\Delta(\mathcal{A}_{\cl Y})\,N\,N^{\star}\,\Delta(\mathcal{A}_{\cl Y})\,\Delta(\mathcal{A}_{\cl Y})\,N]}\\&=\overline{[\Delta(\mathcal{A}_{\cl Y})\,N\,N^{\star}\,\Delta(\mathcal{A}_{\cl Y})\,N]}=\overline{[\Delta(\mathcal{A}_{\cl Y})\,N\,\Delta(\mathcal{A}_{\cl X})]}=M 
 \end{align*}

Using the fact that $\overline{[\Delta(\mathcal{A}_{\cl Y})\,\mathcal{A}_{\cl Y}\,\Delta(\mathcal{A}_{\cl Y})]}=\mathcal{A}_{\cl Y}$, we get $$\overline{[M^{\star}\,\mathcal{A}_{\cl Y}\,M]}=\overline{[N^{\star}\,\Delta(\mathcal{A}_{\cl Y})\,\mathcal{A}_{\cl Y}\,\Delta(\mathcal{A}_{\cl Y})\,N]}=\overline{[N^{\star}\,\mathcal{A}_{\cl Y}\,N]}=\mathcal{A}_{\cl X},$$
and with similar arguments we have that $$\mathcal{A}_{\cl Y}=\overline{[M\,\mathcal{A}_{\cl X}\,M^{\star}]}\,\,,\Delta(\mathcal{A}_{\cl X})=\overline{[M^{\star}\,M]}\,\,,\Delta(\mathcal{A}_{\cl Y})=\overline{[M\,M^{\star}]}$$
We define the TROs $M_2=\begin{pmatrix}E_2 & 0\\

0 & 0\end{pmatrix}\,M\,\begin{pmatrix}D_2 & 0 \\

0 & 0\end{pmatrix}$ and $M_1=\begin{pmatrix}0 & 0\\

0 & E_1\end{pmatrix}\,M\,\begin{pmatrix}0 & 0 \\

0 & D_1\end{pmatrix}$

Since $\cl Y=\overline{[E_2\,Y\,E_1]}$, we have that 

\begin{align*}
 \begin{pmatrix}0 & \cl Y\\
 0 & 0\end{pmatrix}&= \begin{pmatrix}E_2 & 0\\
 0 & 0\end{pmatrix}\,\mathcal{A}_{\cl Y}\,\begin{pmatrix}0 & 0\\
 0 & E_1\end{pmatrix}=\begin{pmatrix}E_2 & 0\\
 0 & 0\end{pmatrix}\,M\,\mathcal{A}_{\cl X}\,M^{\star}\,\begin{pmatrix}0 & 0\\
 0 & E_1\end{pmatrix}\\&=\begin{pmatrix}E_2 & 0\\
 0 & 0\end{pmatrix}\,\left[M\,\Delta(\mathcal{A}_{\cl X})\,M^{\star}+M\,\begin{pmatrix}0 & \cl X\\
 0 & 0\end{pmatrix}\,M^{\star}\right]\,\begin{pmatrix}0 & 0\\
 0 & E_1\end{pmatrix}
\end{align*}

But $M\,\Delta(\mathcal{A}_{\cl X})\,M^{\star}=\Delta(\mathcal{A}_{\cl Y})=\begin{pmatrix}E_2 & 0\\

0 & E_1\end{pmatrix}$, so it holds that $$\begin{pmatrix}E_2 & 0\\

 0 & 0\end{pmatrix}\,M\,\Delta(\mathcal{A}_{\cl X})\,M^{\star}\,\begin{pmatrix}0 & 0\\

 0 & E_1\end{pmatrix}=0,$$

and thus \begin{align*}
 \begin{pmatrix}0 & \cl Y\\
 0 & 0\end{pmatrix}&=\begin{pmatrix}E_2 & 0\\
 0 & 0\end{pmatrix}\,M\,\begin{pmatrix}0 & \cl X\\
 0 & 0\end{pmatrix}\,M^{\star}\,\begin{pmatrix}0 & 0\\
 0 & E_1\end{pmatrix}\\&=\begin{pmatrix}E_2 & 0\\
 0 & 0\end{pmatrix}\,M\,\begin{pmatrix}D_2 & 0\\
 0 & 0\end{pmatrix}\,\begin{pmatrix}0 & \cl X\\
 0 & 0\end{pmatrix}\,\begin{pmatrix}0 & 0\\
 0 & D_1\end{pmatrix}\,M^{\star}\,\begin{pmatrix}0 & 0\\
 0 & E_1\end{pmatrix}=\\ & M_2\,\begin{pmatrix}0 & \cl X\\
 0 & 0\end{pmatrix}\,M_1^{\star}
\end{align*}

Similarly, $\begin{pmatrix}0 & \cl X\\

0 & 0\end{pmatrix}=M_2^{\star}\,\begin{pmatrix}0 & \cl Y\\

0 & 0\end{pmatrix}\,M_1$,

and therefore $$\begin{pmatrix}0 & \cl X\\

0 & 0\end{pmatrix}\sim_{TRO}\begin{pmatrix}0 & \cl Y\\

0 & 0\end{pmatrix}$$.

Since $(\mathbb{C},0)\,\begin{pmatrix}0 & \cl X\\

0 & 0\end{pmatrix}\,\begin{pmatrix}0\\

\mathbb{C}\end{pmatrix}=\cl X$ and $\begin{pmatrix}\mathbb{C}\\

0\end{pmatrix}\,\cl X\,(0,\mathbb{C})=\begin{pmatrix}0 & \cl X\\

0 & 0\end{pmatrix}$.

we have that $X\sim_{TRO}\begin{pmatrix}0 & \cl X\\

0 & 0\end{pmatrix}$.

Similarly, $\cl Y\sim_{TRO}\begin{pmatrix}0 & \cl Y\\

0 & 0\end{pmatrix}$.
Therefore, according to Lemma \ref{lem1}, we get $\cl X\sim_{\Delta}\cl Y$.

\end{proof}

Theorem \ref{forunitops} and Theorem 3.11 in \cite{Elekak} imply the following corollary:

\begin{corollary}

$\Delta$-Morita equivalence of operator spaces is an equivalence relation.

\end{corollary}

\begin{theorem}

\label{the}

Let $\cl X\,,\cl Y$ be operator spaces. The following are equivalent:\\
$i)\,\,\cl X\sim_{\sigma\,\Delta} \cl Y$\\
$ii)$ $\cl X$ and $\cl Y$ are $\sigma\,\Delta$-Morita equivalent\\
$iii)$ $\cl X$ and $\cl Y$ are stably isomorphic.

\end{theorem}

\begin{proof}

$i)\iff ii)$ Check the proof of the Theorem \ref{forunitops}.\\

$i)\implies iii)$ See \cite[Theorem 4.6]{Elekak}.\\

$iii)\implies i)$ Since $\cl X\sim_{st} \cl Y$, we have that $K_{\infty}(\cl X)\cong K_{\infty}(\cl Y)$, but since $\cl X\sim_{\sigma\,TRO} K_{\infty}(\cl X)$, we get $\cl X\sim_{\sigma\,\Delta}K_{\infty}(\cl Y).$ Also, $\cl Y\sim_{\sigma TRO}K_{\infty}(\cl Y)$, so $\cl Y\sim_{\sigma\,\Delta}K_{\infty}(\cl Y)$, and due to the fact that $\sim_{\sigma\,\Delta}$ is an equivalence relation, we have $\cl X\sim_{\sigma\,\Delta} \cl Y.$

\end{proof}

\section{$\Delta$-Morita equivalence of unital operator spaces }

\begin{definition}

We call an operator space $\cl X$ unital if there exists a completely isometric map $\phi:\cl X\to \mathbb{B}(H)$ such that $I_{H}\in\phi(\cl X).$
\end{definition}

If $\cl Y$ is an operator space that is bimodule over the $C^\star$ algebra $\cl A,$ 
we say that the map $$(\pi, \psi, \pi ):
\;_{\cl A}\cl Y_{\cl A}\rightarrow \bb B(H)$$ is a completely contractive bimodule map
 if $\psi: \cl Y\rightarrow \bb B(H)$ is a completely contractive map and $\pi:\cl A\rightarrow \bb B(
 H)$ is a $*-$homomorphism such that
 $$\phi(asb)=\pi(a)\phi(s)\pi(b),\;\;\;\forall\;a,b\;\in\;\cl A,\;\;s\;\in\;\cl Y.$$

\begin{lemma}\label{x} Let $\cl X, \cl Y$ be operator spaces and $M$ be a TRO such that 

$$\cl X=\overline{[M\cl YM^\star]},\;\;\; \cl Y=\overline{[M\cl XM^\star]}.$$ We denote 
$ \cl A=\overline{[M^\star M]}, B=\overline{[MM^\star]}.$ 
For every completely isometric bimodule map $$(\pi, \psi, \pi ):
\;_{\cl A}\cl Y_{\cl A}\rightarrow \bb B(H)$$ there exists a completely isometric bimodule map $$(\sigma, \phi, \sigma ): \;_{\cl B}\cl X_{\cl B}\rightarrow \bb B(K)$$ 
and a TRO $N\subseteq \bb B(H,K)$ such that $$\psi (\cl Y)=
\overline{[N^\star\phi (\cl X)N]}, \;\;\phi (\cl X)=\overline{[N\psi (\cl Y)N^\star ]}, 
\;\;\pi (\cl A)=\overline{[N^\star N]},\;\;\sigma (\cl B)=\overline{[NN^\star ]}.$$ 

\end{lemma}

\begin{proof}Suppose that $K=M\otimes^h _AH$ is the Hilbert space with the inner product given by 
$$\sca{m\otimes \xi , n\otimes \omega }=\sca{\pi (n^\star m)\xi , \omega }
\;\;m,n\,\in\,M, \;\;\xi ,\omega \;\in \;H.$$ 

By the usual arguments, we can define a completely isometric map $\phi : \cl X
\rightarrow \bb B(K)$ given by $$\phi (msn^\star )(l\otimes \xi )=m\otimes 
\psi (sn^\star l)(\xi ), m,n,l \in M, s\in \cl S, \xi \in H$$
and the $*-$homomorphism $\sigma : B\rightarrow \bb B(K)$ given by 
$$\sigma (mn^\star )(l\otimes \xi )=m\otimes \pi (n^\star l)(\xi ).$$

We also define the map 
$\mu : M\rightarrow \bb B(L, K)$ given by 

$$\mu (m)(\pi (n^\star l)(\xi ))=(mn^\star l)\otimes \xi $$ and the map 
$\nu : M^\star \rightarrow \bb B(K, L)$ given by 

$$\nu (m^\star )(l\otimes \xi )=\pi (m^\star l)(\xi ).$$

We can easily see that $\nu (m^\star )=\mu (m)^\star$ for all $m$, and $N=\mu (M)$ is a TRO satisfying 
$$\psi (\cl Y)=\overline {[N^\star\phi (\cl X)N]}, \;\;\phi (\cl X)=
\overline {[N\psi (\cl Y)N^\star ]}, \;\;\pi (A)=\overline {[N^\star N]},\;\;\sigma (B)=
\overline{[NN^\star ]}.$$ 

\end{proof}

\begin{lemma}\label{xx} Let $\cl X\subseteq \bb B(K^\prime ), \cl Y\subseteq \bb B(H^\prime)$ be unital operator spaces and $M\subseteq \bb B(H^\prime , K^\prime )$ be a TRO such that 
$$\cl X=\overline{[M\cl YM^\star ]}, \cl Y=\overline{[M^\star \cl XM]}.$$ We denote 
$ \cl A=\overline {[M^\star M]}, \cl B=\overline {[MM^\star ]}.$ For every completely isometric bimodule map $$(\pi, \psi, \pi ):
\;_{\cl A}\cl Y_{\cl A}\rightarrow \bb B(H)$$ there exists a unital completely isometric 
bimodule map $$(\sigma, \phi, \sigma ): \:_{\cl B}\cl X_{\cl B}\rightarrow \bb B(K)$$ 
and a TRO $N\subseteq \bb B(H,K)$ such that $$\psi (\cl Y)=
\overline{[N^\star \phi (\cl X)N]}, \;\;\phi (\cl X)=\overline{[N\psi (\cl Y)N^\star ]}, 
\;\;\pi (A)=\overline{[N^\star N]},\;\;\sigma (B)=\overline{[NN^\star ]}.$$ 

\end{lemma}

\begin{proof}Suppose that $K, \psi ,\mu ,\sigma, N$ are as in the proof of Lemma \ref{x}. We can see that 
$$\phi (msn^\star )=\mu (m)\psi (s)\mu (n)^\star , \;\;\forall \;m,n\;\in M, \;s\;\in \cl X.$$

Since $\cl Y$ is unital, we have that $$\phi (mn^\star)=
\mu (m)\psi (I_{H^\prime} )\mu (n)^\star , \;\;\forall \;m,n\;\in M.$$

Assume that $$I_{K^\prime }= \lim_\lambda \sum_{i=1}^{k_\lambda }m_i^\lambda (m_i^\lambda )^\star.$$ If $l\in M, \xi \in H$, we have 
$$\phi (I_{K^\prime} )(l\otimes \xi )=\lim_\lambda \sum_{i=1}^{k_\lambda }\mu (m_i^\lambda )\psi (I_{H^\prime}
)\mu ((m_i^\lambda ))^\star (l\otimes \xi ) =$$
$$ \lim_\lambda \sum_{i=1}^{k_\lambda }\mu (m_i^\lambda )\psi (I_{H^\prime})(\pi ((m_i^\lambda )^\star l)(\xi ) =
\lim_\lambda \sum_{i=1}^{k_\lambda }\mu (m_i^\lambda )\psi ((m_i^\lambda )^\star l)(\xi ).$$

We can easily see that $\psi |_{M^\star M}=\pi,$
thus $$\phi (I_{K^\prime} )(l\otimes \xi )=\lim_\lambda \sum_{i=1}^{k_\lambda }\mu (m_i^\lambda )\pi ((m_i^\lambda )^\star l)(\xi )=\lim_\lambda 
(\sum_{i=1}^{k_\lambda }m_i^\lambda (m_i^\lambda )^\star l)\otimes \xi =l\otimes \xi.$$

Therefore, $\phi (I_{K^\prime} )=I_K.$

\end{proof}

\begin{lemma}\label{unit}

Let $\cl X\,,\cl Y$ be unital operator spaces such that $\cl X\sim_{\Delta} \cl Y.$ Then, there exist completely isometric maps $$\phi:\cl X\to \mathbb{B}(H)\,,\psi:\cl Y\to \mathbb{B}(K)$$ such that $I_{H}\in\phi(\cl X)\,,I_{K}\in\psi(\cl Y)$ and a $\sigma$-TRO $L\subseteq \mathbb{B}(K,H)$ such that $$\psi(\cl Y)=\overline{[L^{\star}\,\phi(\cl X)\,L]}\,\,,\phi(\cl X)=\overline{[L\,\psi(\cl Y)\,L^{\star}]}.$$

\end{lemma}

\begin{proof}

We have that $\cl Y\sim _{TRO} K_\infty (\cl Y)$, and the TRO equivalence is implemented by one TRO. Since $\cl X$ and $\cl Y$ are unital, 
by \cite{Elekak} $K_\infty (\cl Y) \cong K_\infty (\cl X) $ as 
$K_\infty -$ operator modules. Lemma \ref{x} implies that there exists a completely isometric map $\zeta : \cl Y\rightarrow \zeta (\cl Y)$
such that $\zeta (\cl Y)\sim _{TRO}K_\infty (\cl X)$, and this TRO equivalence is implemented by one TRO. Since $\cl X\sim _{TRO} K_\infty (\cl X) $
with one TRO as in the proof of Theorem 2.1 in \cite{Ele14}, we have that 
$\zeta (\cl Y)\sim _{TRO}\cl X$ with one TRO. From Lemma \ref{x}, given the complete isometry $\zeta ^{-1}: \zeta (\cl Y)\rightarrow \cl Y$, 
there exists 
a complete isometry 
$\phi : \cl X\rightarrow \phi (\cl X)$ and a TRO $M$ such that 
$$\cl Y=\overline{[M\phi (\cl X)M^\star ]},\;\;\;\phi (\cl X)=\overline{[M^\star \cl YM]}.$$
By Lemma 4.9 in \cite{Elekak}, the algebra $\overline{[M^\star M]}$ is unital, thus $\phi (\cl X)$ is a unital operator space. 
The map $\phi ^{-1}: \phi (\cl X)\rightarrow \cl X$ is a complete isometry, thus by Lemma \ref{xx} 
there exists a unital complete isometry $\psi : \cl Y\rightarrow \psi (\cl Y) $ and a TRO $L$ 
such that $$\psi (Y)=\overline{[L^\star \cl XL]}, \;\;\cl X=\overline{[L\psi (\cl S)L^\star ]}.$$

\end{proof}

If $\cl X$ is an operator space, we denote by $M_{\ell}(\cl X)$ (resp. $M_{r}(\cl X)$) the left (resp. right) multiplier algebra of $\cl X.$ We also denote $$\mathcal{A}_{\ell}(\cl X)=\Delta(M_{\ell}(\cl X))\,\,,\mathcal{A}_{r}(\cl X)=\Delta(M_{r}(\cl X))$$.

\begin{remark}\em{ If we consider $\cl X$ as unital subspace of its $C^{\star}$-envelope, $C^{\star}_{env}(\cl X),$ then by Proposition 4.3 in \cite{Ble01}, we have

$$M_{\ell}(\cl X)=\left\{a\in C^{\star}_{env}(\cl X):a\,\cl X\subseteq \cl X\right\}$$ and

$$M_{r}(\cl X)=\left\{a\in C^{\star}_{env}(\cl X):\cl X\,a\subseteq \cl X\right\}$$}

\end{remark}

\begin{lemma}

\label{lem}

If $\cl X\,,\cl Y$ are $\Delta$-equivalent unital operator spaces, we can consider that $\cl X\subseteq C^{\star}_{env}(\cl X)\subseteq \mathbb{B}(H)\,\,,\cl Y\subseteq C^{\star}_{env}(\cl Y)\subseteq \mathbb{B}(K)$ and there exists a TRO $M\subseteq \mathbb{B}(H,K)$ such that $\,\cl X=\overline{[M^{\star}\,\cl Y\,M]}\,\,,\cl Y=\overline{[M\,\cl X\,M^{\star}]}$ and also

\begin{center}

 \,\,\,\,\,$C^{\star}_{env}(\cl X)=\overline{[M^{\star}\,C^{\star}_{env}(\cl Y)\,M]}\,,C^{\star}_{env}(\cl Y)=\overline{[M\,C^{\star}_{env}(\cl X)\,M]}$

\end{center}

\begin{center}

 $M_{l}(\cl X)=\overline{[M^{\star}\,M_{l}(\cl Y)\,M]}\,\,,M_{l}(\cl Y)=\overline{[M\,M_{l}(\cl X)\,M^{\star}]}$

\end{center}

\begin{center}

 $M_{r}(\cl X)=\overline{[M^{\star}\,M_{r}(\cl Y)\,M]}\,\,,M_{r}(\cl Y)=\overline{[M\,M_{r}(\cl X)\,M^{\star}]}$ 

\end{center}

\end{lemma}

\begin{proof}

From Lemma \ref{unit}, we may assume that $\cl X$ and $\cl Y$ have TRO equivalent completely isometric representations whose images are TRO equivalent by one TRO.
Using this fact and the proof of Theorem 5.10 in \cite{Elekak}, we 
may consider that there exists a TRO $M$ such that 
$$C^{\star}_{env}(\cl X)=\overline{[M^{\star}\,C^{\star}_{env}(\cl Y)\,M]}\,,C^{\star}_{env}(\cl Y)=\overline{[M\,C^{\star}_{env}(\cl X)\,M]}.$$ 
Let us prove that $M_{l}(\cl X)=\overline{[M^{\star}\,M_{l}(\cl Y)\,M]}.$ Let $a\in M_{l}(\cl Y)$, that is $a\in C^{\star}_{env}(\cl Y)$ and $a\,\cl Y\subseteq \cl Y.$ For all $m\,,n\in M$, we have that $m^{\star}\,a\,n\in C^{\star}_{env}(\cl Y)$ and $$m^{\star}\,a\,n\,\cl X=m^{\star}\,a\,n\,M^{\star}\,\cl Y\,M\subseteq m^{\star}\,a\,\cl Y\,M\subseteq M^{\star}\,\cl Y\,M=\cl X,$$
so $m^{\star}\,a\,n\in M_{l}(\cl X)$, that is $M^{\star}\,M_{l}(\cl Y)\,M\subseteq M_{l}(\cl X).$ Similarly, $M\,M_{l}(\cl X)\,M^{\star}\subseteq M_{l}(\cl Y)$, so $$M^{\star}\,M\,M_{l}(\cl X)\,M^{\star}\,M\subseteq M^{\star}\,M_{l}(\cl Y)\,M\subseteq M_{l}(\cl X),$$ but $M^{\star}\,M\,C^{\star}_{env}(\cl X)=C^{\star}_{env}(\cl X)$.

\end{proof}

The proof of the previous Lemma implies the following corollary:

\begin{corollary}

If $\cl X\,,\cl Y$ are $\Delta$-equivalent unital operator spaces, then $M_{l}(\cl X)\sim_{\Delta} M_{l}(\cl Y)$, and thus $M_{l}(\cl X)$ and $M_{l}(\cl Y)$ are stably isomorphic. The same assertion holds for $M_{r}(\cl X)$ and $M_{r}(\cl Y).$

\end{corollary}

\begin{definition}

If $\cl X$ is an operator space, then we define the operator algebra $$\Omega_{\cl X}=\begin{pmatrix}A_{l}(\cl X) & \cl X\\
0 & A_{r}(\cl X)\end{pmatrix}$$

\end{definition}

\begin{theorem}

If $\cl X\,,\cl Y$ are unital operator spaces, the following are equivalent:\\
$i)$ $\cl X$ and $\cl Y$ are stably isomorphic.\\
$ii)$ $\cl X\sim_{\sigma\,\Delta} \cl Y.$\\
$iii)$ $\cl X\sim_{\Delta} \cl Y$\\
$iv)$ $\Omega_{\cl X}$ and $\Omega_{\cl Y}$ are stably isomorphic.\\
$v)$ $\Omega_{\cl X}\sim_{\sigma\,\Delta}\Omega_{\cl Y}$\\
$vi)$ $\Omega_{\cl X}\sim_{\Delta} \Omega_{\cl Y}.$

\end{theorem}

\begin{proof}

We have proved the equivalence $i)\iff ii)$ at the Theorem \ref{the}. Also, $ii)\implies iii)$ is obvious.\\

$iii)\implies ii)$ Suppose that $\phi(\cl X)=\overline{[M^{\star}\,\psi(\cl Y)\,M]}\,\,,\psi(\cl Y)=\overline{[M\,\phi(\cl X)\,M^{\star}]}$ for some TRO $M.$ By Lemma 4.9 in \cite{Elekak}, the $C^{\star}$-algebras $\overline{[M^{\star}\,M]}\,\,,\overline{[M\,M^{\star}]}$ are unital, so it follows that $\cl X\sim_{\sigma\,\Delta} \cl Y.$

Similarly, we have the equivalence $iv)\iff v)\iff vi).$ It remains to prove that $iii)\iff vi).$\\

$iii)\implies vi)$. If $\cl X\sim_{\Delta} \cl Y$, then by Lemma \ref{lem}, there exists a TRO $M$ such that $$\cl X=\overline{[M^{\star}\,\cl Y\,M]}\,\,,\cl Y=\overline{[M\,\cl X\,M^{\star}]}$$ $$M_{l}(\cl X)=\overline{[M^{\star}\,M_{l}(\cl Y)\,M]}\,\,,M_{l}(\cl Y)=\overline{[M\,M_{l}(\cl X)\,M^{\star}]}$$ 

$$M_{r}(\cl X)=\overline{[M^{\star}\,M_{r}(\cl Y)\,M]}\,\,,M_{r}(\cl Y)=\overline{[M\,M_{r}(\cl X)\,M^{\star}]}$$

Since $A_{l}(\cl X)=\Delta(M_{l}(\cl X))\,\,,A_{l}(\cl Y)=\Delta(M_{l}(\cl Y))\,\,,A_{r}(\cl X)=\Delta(M_{r}(\cl X))\,,A_{r}(\cl Y)=\Delta(M_{r}(\cl Y))$, we get $$A_{l}(\cl X)=\overline{[M^{\star}\,A_{l}(\cl Y)\,M]}\,\,,A_{l}(\cl Y)=\overline{[M\,A_{l}(\cl X)\,M^{\star}]}$$

$$A_{r}(\cl X)=\overline{[M^{\star}\,A_{r}(\cl Y)\,M]}\,\,,A_{r}(\cl Y)=\overline{[M\,A_{r}(\cl X)\,M^{\star}]},$$

so \begin{align*}
 \Omega_{\cl X}&=\begin{pmatrix}A_{l}(\cl X) & \cl X\\
 0 & A_{r}(X)\end{pmatrix}=
  \begin{pmatrix}M^{\star}\,A_{l}(\cl Y)\,M & M^{\star}\,\cl Y\,M\\
 0 & M^{\star}\,A_{r}(\cl Y)\,M\end{pmatrix}\\
 &=\begin{pmatrix}M^{\star} & 0\\
 0 & M^{\star}\end{pmatrix}\,\begin{pmatrix}A_{l}(\cl Y) & \cl Y\\
 0 & A_{r}(\cl Y)\end{pmatrix}\,\begin{pmatrix}M & 0\\
 0 & M\end{pmatrix}\\&=
  \begin{pmatrix}M & 0\\
 0 & M\end{pmatrix}^{\star}\,\begin{pmatrix}A_{l}(\cl Y) & \cl Y\\
 0 & A_{r}(\cl Y)\end{pmatrix}\,\begin{pmatrix}M & 0\\
 0 & M\end{pmatrix}
\end{align*}

where $\begin{pmatrix}M & 0\\
0 & M\end{pmatrix}$ is TRO. Similarly, $\Omega_{\cl Y}=\begin{pmatrix}M & 0\\
 0 & M\end{pmatrix}\,\begin{pmatrix}A_{l}(\cl Y) & \cl Y\\
 0 & A_{r}(\cl Y)\end{pmatrix}\,\begin{pmatrix}M & 0\\
 0 & M\end{pmatrix}^{\star}$, and we conclude that $\Omega_{\cl X}\sim_{\Delta}\Omega_{Y}$ (Theorem \ref{main}).\\

$vi)\implies iii)$ Let $\Omega_{\cl X}\sim_{\Delta}\Omega_{\cl Y}.$ The operator algebras $\Omega_{\cl X}\,,\Omega_{\cl Y}$ are $\Delta$-algebraic extensions of $\cl X\,,\cl Y$, respectively, so $\cl X\,,\cl Y$ are $\Delta$-Morita equivalent. According to Theorem \ref{forunitops}, we conclude that $\cl X\sim_{\Delta} \cl Y.$

\end{proof}

\begin{corollary}

If $\cl X\,,\cl Y$ are unital operator spaces, the following are equivalent:\\
$i)$ $\cl X\sim_{\Delta} \cl Y$\\
$ii)$ $\cl X$ and $\cl Y$ are $\Delta$-Morita equivalent.\\
$iii)$ The $\Delta$-pairs $(\Omega_{\cl X},\Delta(\Omega_{\cl X}))\,\,,(\Omega_{\cl Y},\Delta(\Omega_{\cl Y}))$ are $\Delta$-Morita equivalent.

\end{corollary}

\begin{proof}

$i)\iff ii)$ It has been proven previously at Theorem \ref{forunitops}.\\

$iii)\implies ii)$ It is obvious since $\Omega_{\cl X}\,,\Omega_{\cl Y}$ are algebraic $\Delta$-extensions of $\cl X\,,\cl Y$, respectively.\\

$i)\implies iii)$ We may consider, using again the Lemma \ref{lem}, that there exists a TRO $M$ such that $$\cl X=\overline{[M^{\star}\,\cl Y\,M]}\,\,,\cl Y=\overline{[M\,\cl X\,M^{\star}]}$$ $$A_{l}(\cl X)=\overline{[M^{\star}\,A_{l}(\cl Y)\,M]}\,,A_{l}(\cl Y)=\overline{[M\,A_{l}(\cl X)\,M^{\star}]}$$ $$A_{r}(\cl X)=\overline{[M^{\star}\,A_{r}(\cl Y)\,M]}\,\,,A_{r}(\cl Y)=\overline{[M\,A_{r}(\cl X)\,M^{\star}]}$$

Using the TRO $N=\begin{pmatrix}M & 0\\

0 & M\end{pmatrix}$, we have that $\Omega_{\cl X}=\overline{[N^{\star}\,\Omega_{\cl Y}\,N]}\,\,,\Omega_{\cl Y}=\overline{[N\,\Omega_{\cl X}\,N^{\star}]}$. Also, $$\Delta(\Omega_{\cl X})=\begin{pmatrix}A_{l}(\cl X) & 0\\

0 & A_{r}(\cl X)\end{pmatrix}\,\,,\Delta(\Omega_{\cl Y})=\begin{pmatrix}A_{l}(\cl Y) & 0\\

0 & A_{r}(\cl Y)\end{pmatrix},$$
and it is obvious that $\Delta(\Omega_{\cl X})=\overline{[N^{\star}\,\Delta(\Omega_{\cl Y})\,N]}\,\,,\Delta(\Omega_{\cl Y})=\overline{[N\,\Delta(\Omega_{\cl X})\,N^{\star}]}$,
so $\Omega_{\cl X}\sim_{\Delta} \Omega_{\cl Y}$ and $\Delta(\Omega_{\cl X})\sim_{\Delta}\Delta(\Omega_{\cl Y})$ with the same TRO, which means that the $\Delta$-pairs $(\Omega_{\cl X},\Delta(\Omega_{\cl X}))\,\,,(\Omega_{\cl Y},\Delta(\Omega_{\cl Y}))$ are $\Delta$-Morita equivalent.

\end{proof}






\noindent












\end{document}